\documentstyle[amssymb,amsfonts,12pt]{amsart}
\newtheorem{theorem}{Theorem}[section]

\newtheorem{corollary}[theorem]{Corollary}
\newtheorem{proposition}[theorem]{Proposition}

\newtheorem{question}[theorem]{Question}

\theoremstyle{remark}

\newtheorem*{ack*}{Acknowledgment}
\def\QSet{\mbox{\rm\kern.24em
\vrule width.03em height1.48ex depth-.051ex \kern-.26em Q}}

\def\D{{\mathbb D}}\def\x{{\bf x}}\def\y{{\bf y}}\def\z{{\bf z}}\def\0{{\bf 0}}

\def\F{{\mathcal F}}

\def\E{{\mathbb E}}

\def\Z{{\mathbb Z}}
\def\D{{\mathcal D}}

\def\L{{\mathcal L}}
\def\det{{\operatorname{det}}}
\def\gcd{{\operatorname{gcd}}}
\def\bas{\begin{align*}}
\def\eas{\end{align*}}
\def\bi{\begin{itemize}}
\def\ei{\end{itemize}}
\newenvironment{proof}{\noindent {\bf Proof} }{\endprf\par}
\def \endprf{\hfill  {\vrule height6pt width6pt depth0pt}\medskip}
\def\emph#1{{\it #1}}

\begin{document}

\title[Three applications of the Siegel mass formula]{Three applications of the Siegel mass formula}
\author{Jean Bourgain}
\address{School of Mathematics, Institute for Advanced Study, Princeton, NJ 08540}
\email{bourgain@@math.ias.edu}
\author{Ciprian Demeter}
\address{Department of Mathematics, Indiana University, 831 East 3rd St., Bloomington IN 47405}
\email{demeterc@@indiana.edu}

\thanks{The authors are  partially supported by the Collaborative Research NSF grant DMS-1800305}
\begin{abstract}
We present three applications of the Siegel mass formula, using the explicit upper bounds for densities derived in \cite{BoDe}.
\end{abstract}
\maketitle

\section{Background on the Siegel mass formula}
\label{sec4}

Let $m\ge n+1$ and let $\gamma\in M_{m,m}(\Z)$ and $\Lambda\in M_{n,n}(\Z)$ be two positive definite matrices with integer entries. Denote by $A(\gamma,\Lambda)$  the number of  solutions $\L\in M_{m,n}(\Z)$ for
\begin{equation}
\label{Snew27}
\L^*\gamma\L=\Lambda.
\end{equation}
Then Siegel's mass formula \cite{Si} asserts that

\begin{equation}
\label{Snew1863}
A(\gamma,\Lambda)\lesssim_{n,m,\gamma} (\det(\Lambda))^{\frac{m-n-1}{2}}\prod_{p\text{ prime}}\nu_p(\gamma,\Lambda).
\end{equation}
In our forthcoming applications  $m=n+1$ and  $\gamma$ will always  be the identity matrix $I_{n+1}$. In this case, the factor $(\det(\Lambda))^{\frac{m-n-1}{2}}$ is 1.
\bigskip

 In evaluating the densities  $\nu_p(I_{n+1},\Lambda)$ we distinguish two separate cases: $p\nmid\det(\Lambda)$ and $p|\det(\Lambda)$. We recall the following estimate  (Proposition 5.6.2. (ii) in \cite{Yoshi}), see also Proposition 4.2 in \cite{BoDe}.

\begin{proposition}
\label{propnondiv}
We have
\begin{equation}
\label{ababeq6}
\prod_{p\nmid\det (\Lambda)}\nu_p(I_{n+1},\Lambda)\lesssim 1,
\end{equation}
with some universal implicit constant. 

\end{proposition}

Let us next consider the primes $p$ which divide  $\det(\Lambda)$. Recall that the number of such primes is
\begin{equation}
\label{nottoomanyefgryft77856t785687}
O(\frac{\log \det(\Lambda)}{\log\log \det(\Lambda)}).
\end{equation}

 We will denote by $o_p(T)$ the largest $\alpha$ such that $p^\alpha\,|\,T$. We denote by $d(T)$ the number of divisors on $T$, and by $\operatorname{gcd}$ the greatest common divisor.
 If $T$ has factorization 
 $$T=\prod p_i^{\alpha_i}$$ then
 $$d(T)=\prod(1+\alpha_i).$$
 Recall that we have the bound
 $$d(T)\lesssim_\epsilon T^\epsilon.$$

\medskip

For an $n\times n$ matrix $\Lambda$ and for $A,B\subset \{1,\ldots,n\}$ with $|A|=|B|$ we define
$$\mu_{A,B}=\det((\Lambda_{i,j})_{i\in A,j\in B}).$$
\medskip

We recall the following result from \cite{BoDe}.

\begin{proposition}
\label{thecaseofpdivisor}
Let $\Lambda\in M_{n,n}(\Z)$ be a  positive definite matrix and let $p|\det(\Lambda)$. Then
$$\nu_p(I_{n+1},\Lambda)\lesssim \sum_{0\le l_i:1\le i\le n\atop{l_1+l_2+\ldots+l_n\le o_p(\det(\Lambda))}}p^{\beta_2(l_1,\ldots,l_n)+\ldots+\beta_n(l_1,\ldots,l_n)},$$
where $\beta_i=\beta_i(l_1,\ldots,l_n)$ satisfies
$$\beta_i=\min\{(i-1)l_i,(i-2)l_i+\min_{|A|=1}o_p(\mu_{\{1\},A})-l_1,(i-3)l_i+\min_{|A|=2}o_p(\mu_{\{1,2\},A})-l_1-l_2,\ldots,$$$$\ldots,\min_{|A|=i-1}o_p(\mu_{\{1,2,\ldots,i-1\},A})-l_1-l_2-\ldots-l_{i-1}\}$$
\end{proposition}

Let us list two consequences that will be used in the next sections. 
\begin{corollary}[$n=2$]
\label{Con=2}	
Let $\Lambda\in M_{2,2}(\Z)$ be a  positive definite symmetric matrix. Then
$$A(I_{3},\Lambda)\lesssim_\epsilon (\det(\Lambda))^{\epsilon} \operatorname{gcd}(\Lambda_{1,1}, \Lambda_{1,2},\Lambda_{2,2}).$$
	
\end{corollary}
\begin{proof}
Let $p\mid\det(\Lambda)$. Proposition \ref{thecaseofpdivisor} and its symmetric version implies that 
$$\nu_p(I_{3},\Lambda)\lesssim [o_p(\det(\Lambda))]^2p^{\min\{o_p(\Lambda_{1,1}),o_p(\Lambda_{1,2})\}}$$and
$$\nu_p(I_{3},\Lambda)\lesssim [o_p(\det(\Lambda))]^2p^{\min\{o_p(\Lambda_{2,1}),o_p(\Lambda_{2,2})\}}.$$
Combining them leads to 
$$\nu_p(I_{3},\Lambda)\le C [o_p(\det(\Lambda))]^2p^{\min\{o_p(\Lambda_{1,1}),o_p(\Lambda_{1,2}), o_p(\Lambda_{2,2})\}},$$
with $C$ independent of $p$. Thus 
\begin{align*}
\prod_{p\mid\det (\Lambda)}\nu_p(I_{3},\Lambda)&\lesssim C^{O(\frac{\log \det(\Lambda)}{\log\log \det(\Lambda)})}d(\det(\Lambda))^2\operatorname{gcd}(\Lambda_{1,1},\Lambda_{1,2},\Lambda_{2,2})\\&\lesssim_\epsilon (\det(\Lambda))^{\epsilon} \operatorname{gcd}(\Lambda_{1,1}, \Lambda_{1,2},\Lambda_{2,2}).
\end{align*}
The result now follows by combining this inequality with \eqref{Snew1863} and \eqref{ababeq6}. 

\end{proof}
\begin{corollary}[$n=3$]
\label{Con=3}Let $\Lambda\in M_{3,3}(\Z)$ be a  positive definite matrix. Then
$$A(I_{4},\Lambda)\lesssim_\epsilon (\det(\Lambda))^{\epsilon}	
\gcd(\Lambda_{A,B}:\;A,B\subset\{1,2,3\},\;|A|=|B|=2).$$
\end{corollary}	
\begin{proof}
Use the bound $$\beta_2(l_1,l_2,l_3)\le l_2$$ and $$\beta_3(l_1,l_2,l_3)\le \min_{|A|=2}o_p(\mu_{\{1,2\},A})-l_1-l_2$$
(and its symmetric versions). The rest is the same as in Corollary \ref{Con=2}. 

\end{proof}

In the next sections we present three applications of these corollaries.

\section{Uneven Parsell-Vinogradov sums}

Our first application concerns an essentially  sharp estimate for the eigth moment of the  quadratic  Parsell--Vinogradov sums, in the general case when $N$ and $M$ are unrelated. The special case $N\sim M$ was proved in \cite{BoDe1}, as a consequence of the $l^8(L^8)$ decoupling for the surface
$$(t,s,ts,t^2,s^2).$$

\begin{theorem}

	For each $N,M\ge 1$	
	$$\|\sum_{n=1}^N\sum_{m=1}^Me(nx_1+mx_2+nmx_3+n^2x_4+m^2x_5)\|^8_{L^8([0,1]^5)}\lesssim(NM)^{1+\epsilon}(N^3M^3+N^4+M^4).$$	
\end{theorem}	
\begin{proof}
The left hand side represents the number of integral solutions $n_1,\ldots,n_8\in [1,N]$, $m_1,\ldots, m_8\in[1,M]$ of the system
\begin{equation}
\label{fiyure8t58ut85u}
\begin{cases}
n_1+\ldots+n_4=n_5+\ldots+n_8\\
m_1+\ldots+m_4=m_5+\ldots+m_8\\
n_1^2+\ldots+n_4^2=n_5^2+\ldots+n_8^2\\
m_1^2+\ldots+m_4^2=m_5^2+\ldots+m_8^2\\
n_1m_1+\ldots+n_4m_4=n_5m_5+\ldots+n_8m_8
\end{cases}.
\end{equation}
Let us start by proving that, apart from the small multiplicative term $(NM)^\epsilon$, the proposed upper bound is sharp. First, the system \eqref{fiyure8t58ut85u} is easily seen to have $\sim (NM)^4$ trivial solutions, those satisfying $n_i=n_{i+4}$ and $m_i=m_{i+4}$ for each  $1\le i\le 4$.
Second, for each $1\le m\le M$ the number of solutions with all $m_i$ equal to $m$ is given by 
$$\|\sum_{n=1}^Ne(nx+n^2y)\|_{L^8([0,1]^2)}^8\sim N^5.$$ 
Thus, \eqref{fiyure8t58ut85u} has at least $MN^5+NM^5$ solutions.

Let us next prove the upper bound.  For positive integers $A\lesssim N$, $B\lesssim M$, $C\lesssim N^2$, $D\lesssim M^2$, $E\lesssim NM$, we note that  the number of integral solutions of the system 
$$
\begin{cases}
n_1+\ldots+n_4=A\\
m_1+\ldots+m_4=B\\
n_1^2+\ldots+n_4^2=C\\
m_1^2+\ldots+m_4^2=D\\
n_1m_1+\ldots+n_4m_4=E

\end{cases}
$$
is smaller than the number $N_{A,B,C,D,E}$ of solutions of the following system
$$
\begin{cases}
n_1+\ldots+n_4=4A\\
m_1+\ldots+m_4=4B\\
n_1^2+\ldots+n_4^2=16C\\
m_1^2+\ldots+m_4^2=16D\\
n_1m_1+\ldots+n_4m_4=16E

\end{cases}.$$
The number of solutions of \eqref{fiyure8t58ut85u} is clearly dominated by
\begin{equation}
\label{jiut65p0y07u0-8-1}
\sum_{A\lesssim N}\sum_{B\lesssim M}\sum_{C\lesssim  N^2}\sum_{D\lesssim M^2}\sum_{E\lesssim NM}N_{A,B,C,D,E}^2.
\end{equation}
We also note that $N_{A,B,C,D,E}$ is equal to the number $N_{C_1,D_1,E_1}$ of solutions for the system 
\begin{equation}
\label{sy1}
\begin{cases}
n_1+\ldots+n_4=0\\
m_1+\ldots+m_4=0\\
n_1^2+\ldots+n_4^2=C_1\\
m_1^2+\ldots+m_4^2=D_1\\
n_1m_1+\ldots+n_4m_4=E_1
\end{cases}
\end{equation}
where 
$$
\begin{cases}
C_1=16C-{4A^2}\\
D_1=16D-{4B^2}\\
E_1=16E-{4AB}
\end{cases}.
$$
So
$$\sum_{A,B,C,D,E}N_{A,B,C,D,E}^2\lesssim$$ $$\sum_{C_1=O(N^2)\atop{D_1=O(M^2)\atop{E_1=O(NM)}}}N_{C_1,D_1,E_1}^2|\{(A,B,C,D,E):\;C_1=16C-{4A^2},\;
D_1=16D-{4B^2},\;
E_1=16E-{4AB}\}|.$$
\\
\\
Note that for each such $C_1,D_1,E_1$ we have
$$|\{(A,B,C,D,E):\;A=O(N),B=O(M),C_1=16C-{4A^2},
D_1=16D-{4B^2},
E_1=16E-{4AB}\}|$$$$\lesssim NM.$$
It follows that 
$$\sum_{A,B,C,D,E}N_{A,B,C,D,E}^2\lesssim NM\sum_{C_1,D_1,E_1}N_{C_1,D_1,E_1}^2.$$
The system \eqref{sy1} can be rewritten as follows 
$$\begin{cases}
n_1^2+n_2^2+n_3^2+(n_1+n_2+n_3)^2=C_1\\
m_1^2+m_2^2+m_3^2+(m_1+m_2+m_3)^2=D_1\\
n_1m_1+n_2m_2+n_3m_3+(n_1+n_2+n_3)(m_1+m_2+m_3)=E_1
\end{cases}.
$$
This  has at most as many solutions as the system 
\begin{equation}
\label{sy2}
\begin{cases}
n_1^2+n_2^2+n_3^2=C_1\\
m_1^2+m_2^2+m_3^2=D_1\\
n_1m_1+n_2m_2+n_3m_3=E_1
\end{cases}.
\end{equation}
Indeed, this can be seen by changing variables 
$$\begin{cases}
(n_1+n_2,n_2+n_3,n_3+n_1)\mapsto (n_1,n_2,n_3)\\
(m_1+m_2,m_2+m_3,m_3+m_1)\mapsto (m_1,m_2,m_3)
\end{cases}.$$
On the other hand, \eqref{sy2} is equivalent with
$$
\begin{bmatrix}n_1&n_2&n_3\\m_1&m_2&m_3\end{bmatrix}I_3\begin{bmatrix}n_1&m_1\\n_2&m_2\\n_3&m_3\end{bmatrix}=\begin{bmatrix}C_1&E_1\\E_1&D_1\end{bmatrix}.
$$
Let us analyze the number $\tilde{N}_{C_1,D_1,E_1}$ of solutions of this system. We start with the singular case, when $C_1D_1=E_1^2$. In this case the number of solutions can be estimated directly. We split the discussion into two subcases.

First, if both $C_1,D_1\not=0$, we have equality in the Cauchy--Schwarz inequality. Thus,  $(n_1,n_2,n_3)=\lambda(m_1,m_2,m_3)$ for some nonzero rational $\lambda$. Thus, $(n_1,n_2,n_3)=\lambda_1(a,b,c)$ and $(m_1,m_2,m_3)=\lambda_2(a,b,c)$ for some integers $\lambda_1,\lambda_2,a,b,c$ satisfying $\gcd(a,b,c)=1$. We estimate, adding over all admissible  lines
\begin{align*}
\sum_{0<|C_1|\lesssim N^2}\sum_{0<|D_1|\lesssim M^2}\tilde{N}^2_{C_1,D_1,(C_1D_1)^{1/2}}&\le (\sum_{0<|C_1|\lesssim N^2}\sum_{0<|D_1|\lesssim M^2}\tilde{N}_{C_1,D_1,(C_1D_1)^{1/2}})^2\\&\le (\sum_{|a|,|b|,|c|\lesssim \min(N,M)\atop{\gcd(a,b,c)=1}}\frac{N}{\max(a,b,c)}\frac{M}{\max(a,b,c)})^2
\\&\le (\sum_{|a|,|b|,|c|\lesssim \min(N,M)}\frac{NM}{\max(a,b,c)^2})^2\\&\lesssim (NM\min(N,M))^2.
\end{align*}
When $D_1=0$, \eqref{sy2} boils down to one equation $n_1^2+n_2^2+n_3^2=C_1$. We estimate
$$
\sum_{0<|C_1|\lesssim N^2}\tilde{N}^2_{C_1,0,0}\le \sum_{0<|C_1|\lesssim N^2}(C_1)^{1+\epsilon}\lesssim N^{4+\epsilon}
$$

This proves that the contribution from the singular case satisfies
\begin{equation}
\label{jiut65p0y07u0-8-}
NM\sum_{C_1\lesssim N^2}\sum_{D_1\lesssim M^2}\tilde{N}_{C_1,D_1,(C_1D_1)^{1/2}}^2\lesssim_\epsilon (NM)^{1+\epsilon}(N^4+M^4+(NM\min(N,M))^2).
\end{equation}
\smallskip

Let us  assume that $C_1D_1\not= E_1^2$. Corollary \ref{Con=2} shows that 
$$\tilde{N}_{C_1,D_1,E_1}\lesssim_\epsilon (NM)^{\epsilon}\operatorname{gcd}(C_1,D_1,E_1).$$
Fix $\lambda\lesssim \min(N^2,M^2)$. The number of triples $(C_1,D_1,E_1)$ with $C_1\lesssim N^2$, $D_1\lesssim M^2$, $E_1\lesssim NM$, such that $\lambda=\operatorname{gcd}(C_1,D_1,E_1)$ is trivially dominated by
$$ \frac{N^2}{\lambda}\frac{M^2}{\lambda}\frac{NM}{\lambda}=\frac{N^3M^3}{\lambda^3}.$$
Using these, the contribution from the nonsingular case can be dominated by
\begin{align*}
NM\sum_{\lambda\lesssim N^2}\sum_{C_1,D_1,E_1\atop{\operatorname{gcd}(C_1,D_1,E_1)=\lambda}}\tilde{N}_{C_1,D_1,E_1}^2&\lesssim (NM)^{1+\epsilon}\sum_{\lambda\lesssim \min(N^2,M^2)}\sum_{C_1,D_1,E_1\atop{\operatorname{gcd}(C_1,D_1,E_1)=\lambda}}\lambda^2 \\&\lesssim (NM)^{4+\epsilon}\sum_{\lambda\lesssim \min(N^2,M^2)}\frac1\lambda\\&\lesssim_\epsilon (NM)^{4+\epsilon}.
\end{align*}
This together with  \eqref{jiut65p0y07u0-8-} completes the proof of our result, as $\min(NM)^2\le NM.$

\end{proof}
\section{Non-congruent tetrahedra in the truncated lattice $[0,q]^3\cap \Z^3$}
\label{sec9}

Our goal here is to answer the following question asked in \cite{GILP}.
\begin{question}
	Let $T_3([0,q]^3\cap \Z^3)$ denote the collection of all equivalence classes of congruent tetrahedra with vertices in $[0,q]^3\cap \Z^3$. Is there a  $\delta>0$ and some $C>0$, both independent of $q$ such that
	$$\#T_3([0,q]^3\cap \Z^3)\le Cq^{9-\delta}$$
	for each $q>1$?
\end{question}
A positive answer to this question would have implications on producing lower bounds for the Falconer distance-type problem for tetrahedra. We refer to \cite{GILP} for details on this interesting problem.

Here we give a negative answer to this question. We hope that our approach to answering this question will inspire further constructions which might eventually improve the lower bound for the Falconer-type problem.
\begin{theorem}
	\label{mainthmmmm}
	We have for each $\epsilon>0$ and each $q>1$
	$$\#T_3([0,q]^3\cap \Z^3)\gtrsim_{\epsilon} q^{9-\epsilon}.$$
\end{theorem}
Note the following trivial upper bound, which shows the essential tightness of our result
$$\#T_3([0,q]^3\cap \Z^3)\le Cq^{9}.$$
Indeed,  by translation invariance it suffices to fix one vertex at the origin. The upper bound follows since there are $(q+1)^3$ possibilities for each of the remaining three vertices.
\medskip

\begin{proof}[of Theorem \ref{mainthmmmm}]
	
	As mentioned before, we fix one vertex to be the origin $\0=(0,0,0)$. A class of congruent tetrahedra in $T_3([0,q]^3\cap \Z^3)$ can be identified with a matrix $\Lambda\in M_{3,3}(\Z)$. Namely, the congruence class of the tetrahedron with vertices $\0,\x,\y,\z\in [0,q]^3\cap \Z^3$ is represented by the matrix
	$$\Lambda=\begin{bmatrix}\langle \x,\x\rangle &\langle \x,\y\rangle & \langle \x,\z\rangle\\ \langle \y,\x\rangle &\langle \y,\y\rangle & \langle \y,\z \rangle\\ \langle \z,\x\rangle &\langle \z,\y\rangle & \langle \z,\z\rangle\end{bmatrix}.$$
	A tetrahedron is called non-degenerate if $\x,\y,\z$ are linearly independent. We will implicitly assume the congruence classes correspond to non-degenerate tetrahedra.
	
	We seek for an upper bound on the number $N_{\Lambda}$ of integral solutions $\L=(\x,\y,\z)\in (\Z^3)^3$ to the equation
	$$\L^*\L=\Lambda.$$
	This will represent the number of congruent tetrahedra with  side lengths specified by $\Lambda$.
	
	In the numerology from the Section \ref{sec4},  this corresponds to $n=m=3$. To make the theorems in that section applicable we reduce the counting problem to the $m=3,n=2$ case as follows. One can certainly bound $N_\Lambda$ by $q^{\epsilon} N'_\Lambda$,  where $N'_\Lambda$ is the number of integral solutions $\L'=(\x,\y)\in (\Z^3)^2$ satisfying $$(\L')^*\L'=\Lambda'$$
	and $\Lambda'$ is the $2\times 2$ minor of $\Lambda$ obtained from the first two rows and columns of $\Lambda$.
	Indeed, if $\x,\y$ are fixed, the matrix $\Lambda$ forces $\z$ to lie on the intersection of the sphere of radius $\Lambda_{3,3}^{1/2}$ centered at the origin with, say, a sphere centered at $\x$ whose radius is determined only by the entries of $\Lambda$. These radii are $O(q)$, so the resulting circle can only have $O(q^\epsilon)$ points.
	
	Note also that we only care about those $\Lambda'$ for which there exist $\x,\y\in [0,q]^3\cap \Z^3$ linearly independent, such that
	$$\Lambda'=\begin{bmatrix}\langle \x,\x\rangle &\langle \x,\y\rangle\\ \langle \y,\x\rangle &\langle \y,\y\rangle\end{bmatrix}.$$
	This in particular forces $\Lambda'$ to be positive definite.

	Apply now Corollary \ref{Con=2}.
	This will bound $N'_\Lambda$ by
	$$q^\epsilon \operatorname{gcd}(\Lambda_{i,j}:i,j\not=3)\le q^\epsilon \operatorname{gcd}(\Lambda_{1,1},\Lambda_{2,2}).$$
	Thus
	\begin{equation}
	\label{dfjkgy58g0-riu0g5690i}
	N_\Lambda\lesssim_\epsilon q^\epsilon \operatorname{gcd}(\Lambda_{1,1},\Lambda_{2,2}),
	\end{equation}
	for each $\Lambda$ corresponding to a non-degenerate tetrahedron.
	
	Denote by $M_r$  the number of lattice points on the sphere or radius $r^{1/2}$ centered at the origin. In our case $r\le q^2$ so we know that $M_r\lesssim_\epsilon q^{1+\epsilon}$. We need to work with spheres that contain many points. Let
	$$A:=\{r\le q^2:M_r\ge q/2\}.$$
	Since for each $\epsilon>0$ we have $M_r\le C_\epsilon q^{1+\epsilon}$, a double counting argument shows that  $q^3\le  C_\epsilon\#Aq^{1+\epsilon}+\frac12q^2q$. Thus $\#A\gtrsim_\epsilon q^{2-\epsilon}$.
	
	Note that for each $r_i\in A$ there are $\sim M_{r_1}M_{r_2}M_{r_3}$ non-degenerate tetrahedrons with vertices $\x,\y,\z$ on the spheres centered at the origin and with radii $r_1^{1/2},r_2^{1/2},r_3^{1/2}$ respectively. The congruence class of such a tetrahedron contains
	$$\lesssim_\epsilon q^\epsilon \operatorname{gcd}(r_1,r_2)$$
	elements, according to \eqref{dfjkgy58g0-riu0g5690i}.
	
	We conclude that there are at least
	$$\frac{ M_{r_1}M_{r_2}M_{r_3}}{q^{\epsilon}\operatorname{gcd}(r_1,r_2)}$$
	congruence classes generated by such non-degenerate tetrahedra. As distinct radii necessarily give rise to distinct congruence classes, we obtain the lower  bound
	$$\#T_3([0,q]^3\cap \Z^3)\gtrsim_\epsilon\sum_{r_i\in A}\frac{ M_{r_1}M_{r_2}M_{r_3}}{q^{\epsilon}\operatorname{gcd}(r_1,r_2)}\gtrsim_\epsilon q^{3-\epsilon}\sum_{r_1,r_2,r_3\in A}\frac1{\operatorname{gcd}(r_1,r_2)}.$$

	It is immediate that for each integer $d$ there can be  at most $\frac{q^6}{d^2}$ tuples $(r_1,r_2,r_3)\in [0,q^2]^3$, hence also in $A^3$, with $\operatorname{gcd}(r_1,r_2)=d$. Using this observation and the bound $\#(A^3)\ge C_\epsilon q^{6-\epsilon}$, it follows that for each $\epsilon>0$ at least $\frac12\#(A^3)$ among the triples $(r_1,r_2,r_3)\in A^3$ will have  $\operatorname{gcd}(r_1,r_2)\le \frac{10q^\epsilon}{C_\epsilon}$.
	
	This implies that
	$$\sum_{r_1,r_2,r_3\in A}\frac1{\operatorname{gcd}(r_1,r_2)}\gtrsim_\epsilon q^{6-\epsilon},$$
	which finishes the proof of the theorem.
\end{proof}

\section{Distribution of lattice points on caps}
\label{sec8}

Let $n\ge 2$ and $\lambda\ge 1$ be two integers. Define  the lattice points of the sphere
$$\F_{n,\lambda}=\{\xi=(\xi_1,\ldots,\xi_n)\in\Z^n:|\xi_1|^2+\ldots|\xi_n|^2=\lambda\}.$$

It was proved in \cite{BRS} that
$$|\{(\x,\y)\in\Z^3\times\Z^3:\;|\x|^2=|\y|^2=\lambda,\;|\x-\y|<\lambda^{1/4}\}|\lesssim_\epsilon\lambda^{\frac12+\epsilon}.$$
This is a statement on the average distribution of the lattice points on $\F_{3,\lambda}$ in caps of size $\lambda^{1/4}$. Roughly speaking it states that most such caps contain at most $O(\lambda^\epsilon)$ lattice points. It seems reasonable to conjecture that for $n\ge 3$
\begin{equation}
\label{ababeq1}
|\{(\x^1,\ldots,\x^{n-1})\in(\Z^n)^{n-1}:\;|\x^i|^2=\lambda,\;|\x^i-\x^j|<\lambda^{\frac1{2(n-1)}}\text{ for }i\not= j\}|\lesssim_\epsilon\lambda^{\frac{n-2}{2}+\epsilon}.
\end{equation}
We next prove this conjecture for $n=4$.
\smallskip

Note that if $|\x|^2=|\y|^2=\lambda$ then $\x\cdot\y=\lambda-\frac12|\x-\y|^2$. Denote by $X$ the $(n-1)\times n$ matrix with entries $x_{ij}=x^i_j$, the latter being the $j^{th}$ entry of $\x^i$. We can thus bound the left hand side from \eqref{ababeq1} by
\begin{equation}
\label{ababeq3}
\sum_{\Lambda}|\{X\in \Z^{(n-1)\times n}:\;XX^{T}=\Lambda\}|,
\end{equation}
with the sum extending over all symmetric $(n-1)\times (n-1)$ matrices with integer entries of the form
\begin{equation}
\label{ababeq2}\begin{cases}
{\Lambda}_{i,i}=\lambda\;\text{ for }1\le i\le n-1&\\
|\lambda-\Lambda_{i,j}|\le \rho^2\;\text{  for }1\le i\not=j\le n.
\end{cases}\end{equation}
We use $\rho=\lambda^{\frac1{2(n-1)}}$.

Replacing $(\x^1,\ldots,\x^{n-1})$ by $(\x^1,\y^2,\ldots,\y^{n-1})$ with $\y^i=\x^i-\x^1$ for $2\le i \le n-1$, an alternative expression for \eqref{ababeq3} is
\begin{equation}
\label{ababeq4}
\sum_{\Lambda'}|\{X'\in \Z^{(n-1)\times n}:\;X'(X')^{T}=\Lambda'\}|,
\end{equation}
with the sum over symmetric $(n-1)\times (n-1)$ matrices $\Lambda'$ with integer entries of the form
\begin{equation}
\label{ababeq5}\begin{cases}
{\Lambda}_{1,1}'=\lambda&\\
|\Lambda_{i,j}'|\le \rho^2\;\text{  for }i,j\not=1&\\
\Lambda_{i,1}'=\Lambda_{1,i}'=-\frac12\Lambda'_{i,i}\;\text{  for }2\le i\le n-1.
\end{cases}\end{equation}

We will now estimate \eqref{ababeq4} when $n=4$, using the bounds on local densities from Section \ref{sec4}. We assume $\x^1\wedge \x^2\wedge\x^3\not=0$ and leave the other more immediate case to the reader. Note that a typical $\Lambda'$ in our summation has the form
$$\Lambda'=\begin{bmatrix}\lambda&-a&-b\\-a&2a&c\\-b&c&2b\end{bmatrix},$$
with $\det(\Lambda')\not=0$.

Using  Corollary \ref{Con=3} we bound \eqref{ababeq4} by
$$\lesssim_\epsilon\lambda^\epsilon\sum_{|a|,|b|,|c|\le \rho^2}\gcd(\Lambda_{A,B}':\;A,B\subset\{1,2,3\},\;|A|=|B|=2)$$
$$\lesssim_\epsilon\lambda^\epsilon\sum_{|a|,|b|,|c|\le \rho^2}\gcd(a(2\lambda-a),b(2\lambda-b),4ab-c^2)$$
$$\lesssim_\epsilon\lambda^\epsilon \sum_{d\in\D}d|\{(a,b,c):\;|a|,|b|,|c|\le \rho^2,\;d|a(2\lambda-a),\;d|b(2\lambda-b),\;d|4ab-c^2\}|,$$
where $|\D|\lesssim_\epsilon\rho^{2+\epsilon}$ and $\D$ has all elements $O(\rho^4)$.
\medskip

Write $d=d_1d_2$ where $\prod_{p|d_1}p$ divides $\lambda$ and $(d_2,\lambda)=(d_1,d_2)=1$. Let also $d_1^*$ be the smallest number such that $d_1^*|d_1$ and $d_1|(d_1^*)^2$. The Chinese Remainder Theorem shows that $d|a(2\lambda-a)$ determines $a$ modulo $d_1^*d_2$ within at most 2 values. The same holds for $b$. Once $a,b$ are fixed, $c$ is similarly determined modulo $d^*$ where $d^*$ is the smallest number such that $d^*|d$ and $d|(d^*)^2$.  We can refine the estimate for \eqref{ababeq4} as
$$\lambda^\epsilon \sum_{d\in\D}d\sum_{d=d_1d_2\atop{\prod_{p|d_1}p|\lambda}}\sum_{\Delta|d\atop{d|\Delta^2}}\sum_{d_1',d_1''|d_1\atop{d_1|(d_1')^2,(d_1'')^2}}(\frac{\rho^2}{d_1'd_2}+1)(\frac{\rho^2}{d_1''d_2}+1)(\frac{
\rho^2}{\Delta}+1)$$
$$\lesssim_\epsilon\rho^{2+\epsilon}\sum_{d\in\D}d^{1/2}\sum_{d=d_1d_2\atop{\prod_{p|d_1}p|\lambda}}(\frac{\rho^4}{d_1d_2^2}+\frac{\rho^2}{d_1^{1/2}d_2}+1)$$
$$\lesssim_\epsilon \rho^{6+\epsilon}\lesssim_\epsilon \lambda^{1+\epsilon}.$$
This proves the conjectured bound \eqref{ababeq1} for $n=4$.

\end{document}